\documentclass[a4paper,12pt, twoside, reqno]{amsart}
\usepackage{amsthm}
\usepackage{amsmath}
\usepackage{amssymb}
\usepackage{eucal}

\newtheorem{thm}{Theorem}[section]
\newtheorem{coro}[thm]{Corollary}
\newtheorem{prop}[thm]{Proposition}

\theoremstyle{definition}
\newtheorem{rmk}[thm]{Remark}

\catcode`@=11
\catcode`@=12

\let\emptyset=\varnothing

\begin{document}

\title{A GENERAL CONCEPT OF MULTIPLE FIXED POINT FOR MAPPINGS DEFINED ON SPACES WITH A DISTANCE}
\author{Mitrofan M. Choban$^{1}$ and Vasile Berinde$^{2,3}$}

\medskip
\begin{abstract}
Our main aim in this paper is to introduce a general concept of multidimensional fixed point of a mapping in spaces with distance and establish various multidimensional fixed point results. This new concept simplifies the similar notion from [A. Roldan, J. Martinez-Moreno, C. Roldan,   {\it Multidimensional fixed point theorems in partially ordered complete metric
spaces}, J. Math. Anal. Appl. 396 (2012), 536--545]. The obtained multiple fixed point theorems extend,  generalise and unify many related results in literature. 
\end{abstract}

\maketitle

\pagestyle{myheadings} \markboth{Mitrofan M. Choban and Vasile Berinde} {A GENERAL CONCEPT OF MULTIPLE FIXED POINT...}

 \section{Introduction}

 The notion of {\it multidimensional fixed point} emerged naturally from the rich literature devoted to the study of coupled fixed points in the last four decades. The concept of {\it coupled fixed point} itself  has been first introduced and studied by V. I. Opoitsev,  in a series of papers published in the period 1975-1986, see \cite{Op75}-\cite{Op86}, for the case of mixed monotone nonlinear operators satisfying a nonexpansive type condition.  
 
Later, in 1987, Guo and Lakshmikantham  \cite{Guo},  studied  coupled fixed points in connection with coupled quasisolutions of an initial value problem for ordinary differential equations (see also \cite{Guo88}). In 1991, Chen \cite{Chen} obtained coupled fixed point results of $\frac{1}{2}$-$\alpha$-condensing and mixed monotone operators, where $\alpha$ denotes the Kuratowski's measure of non compactness, thus extending some previous results from \cite{Guo} and \cite{Shen}. In the same year, Chang and Ma \cite{Chang91} discussed some existence results and iterative approximation of coupled fixed points for mixed monotone condensing set-valued operators. Next, Chang, Cho and Huang \cite{Chang96} obtained coupled fixed point results of $\frac{1}{2}$-$\alpha$-contractive and generalized condensing mixed monotone operators.

More recently,  Gnana Bhaskar and Lakshmikantham in \cite{Bha} established coupled fixed point results for mixed monotone operators in partially ordered metric spaces in the presence of a Bancah contraction type condition. Essentialy, the results by Bhaskar and Lakshmikantham in \cite{Bha}  combine, in the context of coupled fixed point theory, the main fixed point results previously obtained by Nieto and Rodriguez-Lopez in \cite{Nie06} and  \cite{Nie07}. The last two papers are, in turn, in continuation to a very important fixed point theorem established in the seminal paper of Ran and Reurings \cite{Ran}, which has the merit to combine a metrical fixed point theorem (the contraction mapping principle) and an order theoretic fixed point result (Tarski's fixed point theorem). 

Various applications of the theoretical results in the previous mentioned papers were also given by several authors to: a) Uryson integral equations \cite{Op78}; b) a system of Volterra integral equations  \cite{Chen}, \cite{Chang96}; c) a class of functional equations arising in dynamic programming \cite{Chang91}; d) initial value problems for first order differential equations with discontinuous right hand side \cite{Guo}; e) (two point) periodic boundary value problems \cite{Ber12a}, \cite{Bha}, \cite{Cir}, \cite{Urs}; f) integral equations and systems of integral equations \cite{Agha}, \cite{Algha}, \cite{Aydi}, \cite{BS}, \cite{Gu}, \cite{Hus}, \cite{Shat}, \cite{Sint}, \cite{Xiao}; g) nonlinear elliptic problems and delayed hematopoesis models \cite{Wu}; h) nonlinear Hammerstein integral equations \cite{Sang}; i) nonlinear matrix and nonlinear quadratic equations \cite{AghaFPT}, \cite{BS}; j) initial value problems for ODE \cite{Amini}, \cite{Samet} etc.
 For a very recent account on the developments of coupled fixed point theory, we also refer to \cite{JNCA}.
 
In 2010, Samet and Vetro  \cite{SV} considered a concept of {\it fixed point of m-order} as a natural extension
of the notion of coupled fixed point.  One year later, Berinde and Borcut \cite{BB1} introduced the concept of {\it triple fixed
point} and proved triple fixed-point theorems using mixed monotone mappings, while, in 2012, Karapinar and Berinde \cite{KB}, have studied quadruple fixed points of nonlinear contractions in partially ordered metric spaces. 

After these papers, a substantial number of articles were dedicated to the study of triple fixed point and quadruple fixed point theory. Next, J. Roldan,  Martinez-Moreno and C. Roldan \cite{R1}  introduced a new concept of  {\it fixed point
of $m$-order}, which is also called by various authors  "a multidimensional fixed point", or "an $m$-tuplet
fixed point", or "an $m$-tuple fixed point". For some other very recent results on this topic we refer to \cite{Aga14}, \cite{Aga}, \cite{Al}, \cite{Bo}, \cite{MB}, \cite{BB2}, \cite{Dal},  \cite{K}, \cite{Kar13}, \cite{Kar13a},  \cite{Imdad13}, \cite{Imdad13a}, \cite{Imdad14}, \cite{Lee}, \cite{Mutlu}, \cite{Ola}, \cite{R1}-\cite{Rus}, \cite{Sol}, \cite{Wang}, \cite{ZCS}.

In the present paper, our main aim is to introduce and study a general concept of multidimensional fixed point in the setting of ordered
spaces with distance. This concept simplifies the similar notion from \cite{R1} and allows us to obtain general multiple fixed point theorems that include as particular cases several related results in literature.




This point of view allows us to reduce the multidimensional case of fixed points and coincidence points to 
the one-dimensional case. Note that, the first author who reduced the problem of finding a coupled fixed point of mixed monotone operators to the problem of finding a fixed point of an increasing operator was Opoitsev, see for example  \cite{Op78}. For a more recent similar approach we refer to \cite{Ber11}.

 \section{Preliminaries }

By a space we understand a  topological  $T_0$-space.
We use the terminology from \cite{Eng, GD, RP, C1}.

Let $X$ be a non-empty set and $d : X\times X \rightarrow \mathbb R$ be a mapping such that:

($i_m$) $d(x, y) \geq  0$, for all $x, y \in X$;

($ii_m$) $d(x, y) + d(y,x) = 0$ if and only if $x = y$.

Then $d$ is called a {\it distance} on $X$, while  $(X, d)$ is called a {\it distance space}.

Let $d$ be a distance on $X$ and $B(x,d,r)$ = $\{y \in X: d(x, y) < r\}$ be the {\it ball} with the center $x$
and radius $r > 0$. The set $U \subset X$ is called {\it $d$-open} if for any $x \in U$ there exists $r > 0$
such that $B(x,d,r) \subset U$. The family $\mathcal T(d)$ of all $d$-open subsets is the topology on $X$ 
generated by $d$. A distance space is a {\it sequential space}, i.e., a space for which a set $B \subseteq X$ is closed if and only 
if together with any sequence it contains all its limits \cite{Eng}.
   
  Let $(X, d)$ be a  distance space, $\{x_n\}_{n \in \mathbb N}$ be a sequence in $X$ and $x \in X$. We say that the sequence   $\{x_n\}_{n \in \mathbb N}$ is:

1)  {\it convergent to} $x$ if and only if $lim_{n\rightarrow \infty }d(x, x_n) = 0$. 
We denote this by $x_n\rightarrow x$ or $x = lim_{n\rightarrow \infty }x_n$ (actually, we might denote better  $x \in lim_{n\rightarrow \infty }x_n$);

2)   {\it convergent} if  it converges to some point $x$ in  $X$;

3)   {\it Cauchy} or {\it fundamental} if $lim_{n, m\rightarrow \infty }d(x_n, x_m) = 0$.
    
  A  distance space $(X, d)$ is called {\it complete} if  every
Cauchy sequence in $X$ converges to some point $x$  in $X$.

    Let $X$ be a non-empty set and $d $ be a distance on $X$. Then:  
\begin{itemize}
\item $(X, d)$ is called a {\it symmetric space} and $d$ is called a {\it symmetric} on $X$ if
  
($iii_m$) $d(x, y)$ = $d(y, x)$,  for all $x, y \in X$;
\item  $(X, d)$ is called a {\it quasimetric space} and $d$ is called a {\it quasimetric} on $X$  if

($iv_m$) $d(x, z) \leq d(x, y) + d(y, z)$,  for all $x, y, z \in X$;
\item  $(X, d)$ is called a {\it metric space} and $d$ is called a {\it metric} if  $d$ is a symmetric and a  quasimetric, simultaneously.
\end{itemize}

      Let $X$ be a non-empty set and  $d(x, y) $ be a distance on $X$
   with the following property:

($N$) for each point $x \in X$ and any $\varepsilon > 0$ there exists $\delta = \delta (x,\varepsilon ) > 0$
such that from $d(x, y) \leq \delta $ and $d(y,z) \leq \delta $ it follows $d(x, z) \leq \varepsilon $.

\noindent
Then $(X, d)$ is called an {\it N-distance space} and $d$ is called an {\it N-distance}
on $X$.
If $d$ is a symmetric, then we say that $d$ is a N-symmetric.

Spaces with N-distances were studied by V. Niemyzki  \cite{Ne1, Ne2} and by S. I. Nedev \cite{N}. Clearly, any (quasi) metric space is a N-distance space. If $d$ satisfies uniformly the N-distance condition, that is, 

($F$) for any $\varepsilon > 0$ there exists $\delta = \delta (\varepsilon ) > 0$
such that from $d(x, y) \leq \delta $ and $d(y,z) \leq \delta $ it follows $d(x, z) \leq \varepsilon $,
then  $d$ is called a {\it F-distance}  or a {\it Fr\'echet distance},  while  $(X, d)$ is 
called an {\it F-distance space}. 

Obviously, any F-distance $d$ is an N-distance, too, but the reverse is not true, in general, see Examples 1.1 and 1.2 in \cite{C1}.
If $d$ is a symmetric and a F-distance on a space $X$, then we say that $d$ is a $F$-symmetric.

 \begin{rmk}\label{rm1.1}  If $(X, d)$ is an $F$-symmetric space, then any convergent sequence is a 
Cauchy sequence. For $N$-symmetric spaces and for quasimetric spaces this assertion is not more true.
 \end{rmk}

 If $s > 0$ and $d(x,y) \leq s [d(x,z) + d(z, y)]$ for all points $x, y, z \in X$, then we say that $d$ is an {\it $s$-distance}.
 Any $s$-distance is an $F$-distance.

 A distance space $(X, d)$ is called an {\it H-distance space} if, for any two distinct points $x, y \in X$,
 there exists $\delta  = \delta (x, y) > 0$ such that $B(x, d, \delta ) \cap B(y, d, \delta ) = \emptyset $.
 
 \begin{rmk}\label{rm1.2} A distance space $(X, d)$  is an $H$-distance space if and
only if any convergent sequence  in $X$ has a unique limit point. 
 \end{rmk}
 
We say that  $(X, d)$ is a $C${\it -distance space} or a {\it Cauchy distance space}  if any convergent Cauchy sequence has a unique limit point.
 
  Fix a mapping $\varphi :X \longrightarrow X$. For any point $x \in X$ we put $\varphi ^0(x) = x$,
  $\varphi ^1(x) = \varphi (x), ..., \varphi ^n(x) = \varphi (\varphi ^{n-1}(x))$,\dots. The
  sequence 
  $$O(\varphi ,x)= \{x_n = \varphi ^n(x): n \in \mathbb N\}$$ is called
  the {\it orbit of $\varphi $} at the point $x$ or the {\it Picard sequence}  at the point $x$. 
 
    Let  $(X, d)$ be a distance space and $\varphi :X \longrightarrow X$  a mapping.  
 We say that the mapping $\varphi $: 
 
-  is {\it contractive} if $d(\varphi (x), \varphi (y)) < d(x,y)$ provided $d(x,y) > 0$;

- is a {\it  contraction} if there exists $\lambda \in [0, 1)$ such that 
$d(\varphi (x), \varphi (y)) \leq \lambda d(x,y)$, for all $x, y \in X$;

-  is {\it strongly asymptotically regular} if
$lim_{n\rightarrow \infty }(d(\varphi ^n(x),\varphi ^{n+1}(x) + d(\varphi ^{n+1}(x), \varphi ^n(x)) ) )$ = $0$, for any $x \in X$.



Now, let $(X,d)$ be a distance space and $m \in \mathbb N $ = $\{1, 2, ...\}$. On the set $X^m$ we consider the distances
$$d^m((x_1,..., x_m), (y_1,...,y_m)) = sup \{d(x_i,y_i): i \leq m\}$$ 
and 
$$\bar{d}^m((x_1,..., x_m), (y_1,...,y_m)) = \sum_{i=1}^{m}  d(x_i,y_i).$$
Obviously, $(X^m, d^m)$ and $(X^m, \bar{d}^m)$ are distance spaces, too.

\begin{prop}\label{P2.1}    Let $(X, d)$ be a distance space. Then:

1. If $d$ is a symmetric, then $(X^m, d^m)$ and $(X^m, \bar{d}^m)$ are  symmetric spaces, too.

2. If $d$ is a quasimetric, then  $(X^m, d^m)$ and $(X^m, \bar{d^m})$ are  quasimetric spaces, too.

3.  If $d$ is a metric, then   $(X^m, d^m)$ and $(X^m, \bar{d}^m)$ are  metric spaces, too.

4.  If $d$ is an $F$-distance space, then  $(X^m, d^m)$ and $(X^m, \bar{d}^m)$ are   $F$-distance spaces, too.

5.   If $d$ is an $N$-distance space, then $(X^m, d^m)$ and $(X^m, \bar{d}^m)$ are   $N$-distance spaces, too.

6.  If $d$ is an $H$-distance space, then $(X^m, d^m)$ and $(X^m, \bar{d}^m)$ are   $H$-distance spaces, too.

7. If  $(X, d)$ is a $C$-distance space, then   $(X^m, d^m)$ and $(X^m, \bar{d}^m)$ are   $C$-distance spaces, too.

8.  If $(X, d)$ is a complete distance space, then $(X^m, d^m)$ and $(X^m, \bar{d}^m)$ are   complete distance spaces, too.

9.  If $d$ is an $s$-distance space, then $(X^m, d^m)$ and $(X^m, \bar{d}^m)$ are   $s$-distance spaces, too.

10. The spaces  $(X^m, d^m)$ and $(X^m, \bar{d}^m)$ share the same convergent sequences and the same Cauchy sequences.
Moreover, the distances $d^m$ and $\bar{d^m}$ are uniformly equivalent, i.e., for each $\varepsilon  > 0$, there exists $\delta $
= $\delta (\varepsilon ) > 0$ such that:

- from $d^m(x,y) \leq \delta $ it follows $\bar{d^m}(x,y) \leq \varepsilon $;

-  from $\bar{d^m}(x,y) \leq \delta $ it follows $d^m(x,y) \leq \varepsilon $. 
\end{prop}    

\begin{proof}  It is well known.
 \end{proof}

 \section{Multiple fixed point principles}

Fix $m \in \mathbb N$ and denote by $\lambda $ = $(\lambda _1,... , \lambda _m)$ a collection of mappings
$\{\lambda _i: \{1, 2, ..., m\}$ $\longrightarrow \{1, 2, ..., m\}: 1\leq i \leq  m\}$.

Let $(X, d)$ be a distance space and $F: X^m \longrightarrow X$ be an operator. The operator $F$ and the
mappings $\lambda $ generate the operator $\lambda F: X^m \longrightarrow X^m$, where 
$$\lambda F(x_1,....,x_m)
= (y_1,...,y_m) \textnormal{ and } y_i = F(x_{\lambda _i(1)}, ..., x_{\lambda _i(m)}),
$$ 
for any point $(x_1, ..., x_m) \in X^m$
and any index $i \in \{1, 2, ..., m\}$. 

A point $a$ = $(a_1, ..., a_m) \in X^m$ is called a $\lambda ${\it -multiple fixed point} of the operator $F$ if $a$ = $\lambda F(a)$,
i.e., $a_i$ = $F(a_{\lambda _i(1)}, ..., a_{\lambda _i(m)})$ for any  $i \in \{1, 2, ..., m\}$.

 We say that the operator $F$: 
 \begin{itemize}
\item is {\it $\lambda $-contractive} if $d^m(\lambda F(x), \lambda F(y)) < d^m(x,y)$, for all $x, y \in X^m$ with $d^m(x,y) > 0$;
\item is a {\it   $\lambda $-contraction} if there exist a number  $k  \in [0, 1)$ such that 
\begin{center}
$d(F(x_1,....,x_m),F(y_1,....,y_m)) \leq k \sup \{d(x_i,y_i): i \leq m\}$, 
\end{center}
for all $(x_1,....,x_m), (y_1,....,y_m) \in X^m$. 
\item is {\it $\bar{\lambda }$-contractive} if $\bar{d}^m(\lambda F(x), \lambda F(y)) < \bar{d}^m(x,y)$,  for all $x, y \in X^m$ with $\bar{d}^m(x,y) > 0$;
\item is a {\it   $\bar{\lambda }$-contraction} if there exists a number  $k  \in [0, 1)$ such that 
\begin{center}
$d(F(x_1,....,x_m),F(y_1,....,y_m)) \leq \frac{k}{m} \cdot \sum_{i=1}^{m} d(x_i,y_i),$ 
\end{center}
for all $(x_1,....,x_m), (y_1,....,y_m) \in X^m$.
\end{itemize}

\begin{prop}\label{P3.1}    Let $(X, d)$ be a distance space, $m \in \mathbb N$, $F: X^m \rightarrow X$ be an operator,
$\lambda=\{\lambda _i: \{1, 2, ..., m\}$ $\longrightarrow \{1, 2, ..., m\}: 1\leq i \leq  m\}$ be a collection of mappings, $k \geq  0$, $a = (a_1,....,a_m) \in X^m$,
$b = (b_1,....,b_m)  \in X^m$ such that
$$d(F(a_{\lambda _i(1)},....,a_{\lambda _i(m)}),F(b_{\lambda _i(1)},....,b_{\lambda _i(m)})) \leq k \sup \{d(a_i,b_i): i \leq m\},$$
for each $1\leq i \leq  m$. 
Then $d^m(\lambda F(a), \lambda F(b)) \leq  k d^m(a,b)$.
\end{prop}    

\begin{proof} Let $u_i$ = $F(a_{\lambda _i(1)},....,a_{\lambda _i(m)})$ and $v_i$ = $F(b_{\lambda _i(1)},....,b_{\lambda _i(m)})$
for any $i \leq  m$. Then $\lambda F(a)$ = $u$ = $(u_1, ..., u_m)$ and  $\lambda F(b)$ = $v$ = $(v_1, ..., v_m)$.   
We have  $d^m(\lambda F(a), \lambda F(b))$ = $d^m(u, v)$ = $\sup \{d(u_i,v_i): i \leq m\}$ =
  $\sup \{d(F(a_{\lambda _i(1)}, ..., a_{\lambda _i(m)})$, $F(b_{\lambda _i(1)}, ..., b_{\lambda _i(m)}))$: $i \leq m\}  \leq $
  $ \sup \{ k \sup\{d(a_{\lambda _i(j)}$, $b_{\lambda _i(j)}): j \leq m\}: i \leq m\}$ 
$\leq k \sup \{d(a_i, b_i): i \leq m\}$ = $k d^m(a, b)$.
 \end{proof}

\begin{coro}\label{3C.1}    Let $(X, d)$ be a distance space $m \in \mathbb N$ and $F: X^m \rightarrow X$ be an operator. 
If $F$ is a $\lambda $-contraction, then $\lambda F$ is a contraction on the distance space  $(X^m, d^m)$.
\end{coro}     
 
\begin{prop}\label{P3.2}    Let $(X, d)$ be a distance space, $m \in \mathbb N$ and $F: X^m \rightarrow X$ be an operator,
$\{\lambda _i: \{1, 2, ..., m\}$ $\longrightarrow \{1, 2, ..., m\}: 1\leq i \leq  m\}$ be a collection of mappings, $k \geq  0$, $a = (a_1,....,a_m) \in X^m$,
$b = (b_1,....,b_m)  \in X^m$ such that 
$$d(F(a_{\lambda _i(1)},....,a_{\lambda _i(m)}),F(b_{\lambda _i(1)},....,b_{\lambda _i(m)})) \leq k/m \sum_{i=1}^{m} d(a_i,b_i),$$
for each $i \in  \{1, 2, ..., m\}$.
If the mapping $\lambda _i$ is a surjection
or, more generally, if $|\cup \{\lambda _i^{-1}(j): j \leq m\}| = m$ for each $i \in  \{1, 2, ..., m\}$, then 
$$\bar{d}^m(\lambda F(a), \lambda F(b)) \leq  k \bar{d}^m(a,b).$$
\end{prop}    

\begin{proof}  We put $u =(u_1,....,u_m) = \lambda F(a)$ and 
 $v =(v_1,....,v_m) = \lambda F(b)$. Then 
 $$\bar{d}^m(\lambda F(a), \lambda F(b)) = \sum_{i=1}^{m} d(u_i,v_i) =$$
  $$=\sum_{i=1}^{m} d(F(a_{\lambda _i(1)}, ..., a_{\lambda _i(m)}),F(b_{\lambda _i(1)}, ..., b_{\lambda _i(m)})) \leq $$
  $$\leq \sum_{i=1}^{m}  k/m \sum_{j=1}^{m}  d(a_{\lambda _i(j)}, b_{\lambda _i(j)})\leq k \sum_{i=1}^{m} d(a_i, b_i)=k \bar{d}^m(a, b).$$
 \end{proof}

\begin{coro}\label{C3.2}     Let $(X, d)$ be a distance space, $m \in \mathbb N$ and $F: X^m \rightarrow X$ be an operator. 
If $F$ is a $\bar{\lambda }$-contraction and for any $i \in  \{1, 2, ..., m\}$ the mapping $\lambda _i$ is a surjection
or, more generally, $|\cup \{\lambda _i^{-1}(j): j \leq m\}| = m$ for each $i \in  \{1, 2, ..., m\}$, then $\lambda F$ is a 
contraction on the distance space  $(X^m, \bar{d}^m)$.
\end{coro}

 \section{Multiple fixed points of general operators}

 Fix $m \in \mathbb N$, a  distance space $(X, d)$, an operator $\varphi :X^m \rightarrow X$
and the mappings $\lambda $ = $\{\lambda _i: \{1, ..., m\} \rightarrow  \{1, ..., m: i \leq m\}$.  
For any point $a$ = $(a_1, ..., a_n)\in X^m$ we put $a(1)$ = $\lambda F(a)$ and $a(n+1)$ = $\lambda F(a(n))$ 
for each  $n \in \mathbb N$.  The sequence $O(F,\lambda ,a)$ = $\{a(n): n \in \mathbb N\}$ is the Picard sequence at the point $a$
relatively to the operator $\lambda F$. 
The orbit $O(F,\lambda ,a)$ is called $(F,\lambda )${\it -bounded} if 
$$\sup \{d^m(a,a(n)) + d^m(a(n),a): n \in \mathbb N\} < \infty. $$
(this is equivalent
to  $$\sup \{\bar{d}^m(a,a(n)) + \bar{d}^m(a(n),a): n \in \mathbb N\}< \infty.) $$ The space $(X, d)$ is called $(F,\lambda )${\it -bounded} if any Picard
sequence $O(F,\lambda ,a)$ is  $(F,\lambda )$-bounded.

\begin{prop}\label{P4.1}   Let $(X, d)$ be a $C$-distance space. Then:

1. $d(x,y)$ = 0 if and only if $x= y$.

 2.  If, for $a \in X^m$, the Picard sequence $O(F,\lambda ,a)$  = $\{a(n): n \in \mathbb N\}$ is a convergent Cauchy sequence 
and $lim_{n\rightarrow \infty }a_n$ = $b$ = $(b_1, ..., b_m)$, then $b$ is a multidimensional fixed point of the operator $F$ with respect 
to the mappings $\lambda $, i.e., 
$$b_i = F(b_{\lambda _i(1)}, ..., b_{\lambda _i(m)}) \textnormal{ for each } i \in  \{1, 2, ..., m\}.
$$
\end{prop}    

\begin{proof}  Assume that $x, y $ are two distinct points of $X$ and $d(x,y) = 0$. Then the points $x, y$ are both limits
of the Cauchy sequence $\{y_n = y: n \in \mathbb N\}$, a contradiction. So, assertion 1 is proved. 

In the conditions of assertion 2, we have $\lambda F(b)$ = $b$.
 \end{proof}

\begin{coro}\label{C4.1}    Let $(X, d)$ be a complete $C$-distance space, $\rho  \in \{d^m, \bar{d}^m\}$, $k > 0$ 
and $F : X^m \longrightarrow X$ be an operator with the following properties:
 \begin{enumerate}
\item there exists $k > 0$ such that $d(F(x), F(y)) < k \rho (x,y)$ for all distinct points $x, y \in X^m$;
\item if $x \in X^m$, then
   the  Picard sequence  $\{x_n \in X: n \in \mathbb N\}$, of $F$ at the point $x$, is a  Cauchy sequence.
\end{enumerate}
     Then:
     
     1. The operators $F$ and $\lambda F$ are continuous.
     
     2. The set $Fix (F)$ of the multidimensional fixed points of $F$ is closed in $X^m$ and non-empty.
     
     3. If  $k \leq 1$, then $F $ has a unique  multidimensional fixed point. \end{coro}

\begin{thm}\label{T4.1}      Let $(X, d)$ be a $(F,\lambda )$-bounded complete $C$-distance space and $F : X^m \longrightarrow X$ be an operator.

1. If $F$ is a $\lambda $-contraction,  then any Picard sequence of the operator $\lambda F$ is a convergent Cauchy sequence and 
$F $ has a unique  multidimensional fixed point.

2. If $F$ is a $\bar{\lambda }$-contraction and for any $i \in  \{1, 2, ..., m\}$ the mapping $\lambda _i$ is a surjection
or, more generally, $|\cup \{\lambda _i^{-1}(j): j \leq n\}| = m$ for each $i \in  \{1, 2, ..., m\}$, then  any Picard sequence 
of the operator $\lambda F$ is a convergent Cauchy sequence and 
$F $ has a unique  multidimensional fixed point.
\end{thm}    

\begin{proof}   Let $\rho $ = $d^m$ in the conditions of Assertion 1 and  $\rho $ = $\bar{d^m}$ in the conditions of Assertion 2.
From Propositions \ref{P3.1} and \ref{P3.2}, respectively, it follows that $\lambda F$ is a contraction on the complete
$C$-distance space $(X^m, \rho )$. Proposition 3.4 from  \cite{C1} ensures that the operator $\lambda F$ has a unique
fixed point which is a  multidimensional fixed point of $F$.  
 \end{proof}
 
 \begin{thm}\label{T4.2}     Let $(X, d)$ be an $N$-symmetric space and $F : X^m \longrightarrow X$ be an operator. 
 
1. If $F$ is a  $\lambda $-contractive operator and  for each point $x \in X^m$ the  Picard sequence  $O(F,\lambda , x)$ = 
$\{x_n : n \in \mathbb N\}$ 
has an accumulation point and  $\lim\limits_{n\rightarrow \infty }d^m(x_n,x_{n+1}) = 0$,  then any Picard sequence of the operator $\lambda F$ is 
a convergent Cauchy sequence and 
$F $ has a unique  multidimensional fixed point.

2. If $F$ is a $\bar{\lambda }$-contractive operator and, for any $i \in  \{1, 2, ..., m\}$, the mapping $\lambda _i$ is a surjection
or, more generally, $|\cup \{\lambda _i^{-1}(j): j \leq m\}| = m$ for each $i \in  \{1, 2, ..., m\}$ and for each point $x \in X^m$ the  Picard sequence  $O(F,\lambda , x)$ = 
$\{x_n: n \in \mathbb N\}$ 
has an accumulation point and  $\lim\limits_{n\rightarrow \infty }\bar{d^m}(x_n,x_{n+1}) = 0$,  then any Picard sequence of the operator $\lambda F$ is 
a convergent Cauchy sequence and 
$F $ has a unique  multidimensional fixed point.
  
3. $d$ is an $H$-distance and any Picard sequence has a unique accumulation point.
\end{thm}

\begin{proof}  Assertion 3 follows immediately from Theorem 4.1 from \cite{C1}. Let $\rho $ be the symmetric constructed in the
proof of Theorem \ref{T4.1}. Then $\lambda F$ is a  strongly asymptotically regular contractive mapping on the $N$-symmetric space $(X^m, \rho )$ and, 
for each point $x \in X^m$, the  Picard sequence  $O(F, \lambda , x)$ 
has an accumulation point. Now, Theorem 4.1 from \cite{C1} completes  the proof.
 \end{proof}

\begin{coro}\label{C4.1} Let $(X, d)$ be an $N$-symmetric compact space and $F : X^m \longrightarrow X$ be an operator.

1. If $F$ is a $\lambda $-contraction,  then any Picard sequence of the operator $\lambda F$ is a convergent Cauchy sequence and 
$F $ has a unique  multidimensional fixed point.

2. If $F$ is a $\bar{\lambda }$-contraction and for any $i \leq m$ the mapping $\lambda _i$ is a surjection
or, more generally, $|\cup \{\lambda _i^{-1}(j): j \leq m\}| = m$, for each $i   \leq m$, then  any Picard sequence 
of the operator $\lambda F$ is a convergent Cauchy sequence and 
$F $ has a unique  multidimensional fixed point.   
 \end{coro}

  The problem of the existence of fixed points for contracting mappings on $F$-symmetric spaces was first studied in \cite{BC2}. 
The following statement improves the fixed point theorems of S. Czerwik \cite{Cz} and 
I. A. Bakhtin \cite{Ba} (see also \cite{RP}). 

\begin{thm}\label{T4.3}     Let $(X, d)$  be a complete $s$-distance symmetric space and $F : X^m \longrightarrow X$ be an operator. 

1. If $F$ is a $\lambda$-contraction,  then any Picard sequence of the operator $\lambda F$ is a convergent Cauchy sequence and 
$F $ has a unique  multidimensional fixed point.

 2. If $F$ is a $\bar{\lambda }$-contraction and, for any $i \leq m$, the mapping $\lambda _i$ is a surjection
or, more generally, $|\cup \{\lambda _i^{-1}(j): j \leq m\}| = m$, for each $i   \leq m$, then  any Picard sequence 
of the operator $\lambda F$ is a convergent Cauchy sequence and 
$F $ has a unique  multidimensional fixed point.
\end{thm}

\begin{proof} Let $\rho $ be the symmetric constructed in the
proof of Theorem  \ref{T4.1}. By virtue of Proposition \ref{P2.1}, $\rho $ is a symmetric $s$-distance. 
Then $\lambda F$ is a  contractive mapping of the $s$-symmetric space $(X^m, \rho )$. Now, Theorem 4.2 from \cite{C1} completes  the proof. 
 \end{proof}

\section{Some particular cases and conclusions}

If we take concrete values of $m\in \mathbb{N}$ and consider various particular functions  $\lambda $ = $\{\lambda _i: \{1, ..., m\} \rightarrow  \{1, ..., m\}: 1\leq i \leq m\}$ then, most of the results in literature dedicated to coupled, triple, quadruple,... fixed point theory, are  obtained as particular cases of the multiple fixed point theorems established in the present paper. 

For example, if $m=2$, $\lambda_1(1)=1$, $\lambda_1(2)=2$; $\lambda_2(1)=2$, $\lambda_2(2)=1$, by our main results we obtain  the coupled fixed point theorems in \cite{Bha} and in various subsequent papers, see especially the singular paper \cite{Sab}, where the setting is a (cone) metric space without any order relation. 

If $m=3$, $\lambda_1(1)=1$, $\lambda_1(2)=2$, $\lambda_1(3)=3$; $\lambda_2(1)=2$, $\lambda_2(2)=1$, $\lambda_2(3)=2$; $\lambda_3(1)=3$, $\lambda_3(2)=2$, $\lambda_3(3)=1$, then the concept of multiple fixed point studied in the present paper reduces to that of triple fixed point, first introduced in \cite{BB1} and intensively studied in many other research works emerging from it.  

We note that, as pointed out in \cite{Sha},  the notion of tripled fixed point due to Berinde and Borcut \cite{BB1} is different from the one defined by Samet and Vetro \cite{SV} for $m = 3$, since in the case of ordered metric spaces in order to keep the mixed monotone property working, it was necessary to take $\lambda_2(3)=2$ and not $\lambda_2(3)=3$.

We mention one more important particular case, i.e., the one of fixed point of $N$-order introduced and studied in  \cite{SV}, which is obtained as particular case of our concept introduced in the present paper, by taking $m=N$, $\lambda_1=$ the identity permutation of $\{1,2,\dots,N\}$ and, for $i\geq 2$, $\lambda_i$ is the cyclical permutation of $\{1,2,\dots,N\}$ that starts with $\lambda_i (1)=i$, i.e., for example, $\lambda_2(1)=2$, $\lambda_2(2)=3$,\dots,$\lambda_2(N-1)=N$, $\lambda_2(N)=1$. Note that in this case the family of mappings  $\lambda $ = $\{\lambda _i: \{1, ..., N\} \rightarrow  \{1, ..., N: i \leq N\}$ satisfies both alternative conditions imposed in Theorems \ref{T4.1}, \ref{T4.2}, \ref{T4.3}, Proposition \ref{P3.2} and Corollaries \ref{C3.2}, \ref{C4.1}, i.e., $\lambda _i$ is a surjection and $|\cup \{\lambda _i^{-1}(j): 1\leq j \leq N\}| = N$, for each $i   \leq N$. 

For other concepts of multiple fixed points considered in literature the condition " $\lambda _i$ is a surjection, for each $i   \leq m$" is no more valid, see for example \cite{BB1} and the research papers emerging from it, while the second condition, $|\cup \{\lambda _i^{-1}(j): 1\leq j \leq m\}| = m$, for each $i   \leq m$, is satisfied.

As the great majority of the papers dealing with coupled, triple, quadruple,..., multiple fixed points  were established in ordered metric spaces or generalised order metric spaces, we shall study them separately in a forthcoming paper \cite{ChoB}, where the basic setting will be an ordered distance space.

We point out the fact that the main idea of this paper was to obtain general multiple fixed point theorems by reducing this problem to a unidimensional fixed point problem and by simultaneously working in a more general and very reliable setting, that of distance spaces. Many other related and relevant results could be obtained in the same way, by reducing the multidimensional fixed point problem to many other independent unidimensional fixed point principles, like the ones established in  \cite{Ber13},  \cite{Ber09},  \cite{Ber09a},  \cite{Ber10},  \cite{Ber11a},  \cite{Ber12},  \cite{BC1},  \cite{BP},  \cite{BerP15} etc.

\section*{Acknowledgements}

This paper has been finalised during the second author's visit to Department of Mathematics and Statistics, King Fahd University of Petroleum and Minerals, Dhahran, Saudi Arabia, in the period December 2016-January 2017. He  would like to acknowledge the support provided by the Deanship of Scientific Research at King Fahd University of Petroleum and Minerals for funding this work through the projects IN151014 and IN151017.

\vskip 0.25 cm {\it $^{a}$ Department of Physics, Mathematics and Information
Technologies\newline
\indent Tiraspol State University\newline 
\indent  Gh. Iablocikin 5., MD2069 Chi\c sin\u au, Republic of Moldova

E-mail: mmchoban@gmail.com }

\vskip 0.5 cm {\it $^{b}$ Department of Mathematics and Computer Science

North University Center at Baia Mare 

Technical University of Cluj-Napoca

Victoriei 76, 430072 Baia Mare ROMANIA

E-mail: vberinde@cunbm.utcluj.ro}

\vskip 0.25 cm {\it $^{c}$ Department of Mathematics and Statistics

King Fahd University of Petroleum and Minerals

Dhahran, Saudi Arabia

E-mail: vasile.berinde@gmail.com}
\end{document}